\documentclass[12pt]{amsart}
 
\usepackage{amssymb, amscd, txfonts}
\usepackage{graphicx}
\usepackage[all]{xy} 

\numberwithin{equation}{section}

\newtheorem{theorem}{Theorem}[section]
\newtheorem{proposition}[theorem]{Proposition}
\newtheorem{lemma}[theorem]{Lemma}

\theoremstyle{definition}

\newtheorem{example}[theorem]{Example}

\theoremstyle{remark}
\newtheorem{remark}[theorem]{Remark}

\renewcommand{\ker}{\operatorname{Ker}}

\newcommand{\Z}{\mathbb{Z}}
\newcommand{\Q}{\mathbb{Q}}

\newcommand{\C}{\mathbb{C}}
\newcommand{\proj}{{\mathbb P}}
\newcommand{\QZ}{\mathbb{Q}/\mathbb{Z}}
\newcommand{\Zn}{\mathbb{Z}/n}

\newcommand{\HH}{\mathcal{H}}

\begin{document}

\title[]{Unramified cohomology, integral coniveau filtration and Griffiths group}
\author[]{Shouhei Ma}
\thanks{Supported by JSPS KAKENHI 17K14158 and 20H00112.} 
\address{Department~of~Mathematics, Tokyo~Institute~of~Technology, Tokyo 152-8551, Japan}
\email{ma@math.titech.ac.jp}
\subjclass[2010]{14F43, 14C15, 14C25}
\keywords{unramified cohomology, integral coniveau filtration, Griffiths group, 
$\mathcal{H}$-cohomology, birational invariant} 

\begin{abstract}
We prove that the $k$-th unramified cohomology $H^{k}_{nr}(X, {\QZ})$ 
of a smooth complex projective variety $X$ with small ${\rm CH}_{0}(X)$ 
has a filtration of length $[k/2]$, 
whose first piece is the torsion part of the quotient of $H^{k+1}(X, {\Z})$ by its coniveau $2$ subgroup, 
and whose next graded piece 
is controlled by the Griffiths group ${\rm Griff}^{k/2+1}(X)$ when $k$ is even 
and is related to the higher Chow group ${\rm CH}^{(k+3)/2}(X, 1)$ when $k$ is odd. 
The first piece is a generalization of the Artin-Mumford invariant ($k=2$) and 
the Colliot-Th\'el\`ene--Voisin invariant ($k=3$). 
We also give an analogous result for the $\mathcal{H}$-cohomology 
$H^{d-k}(X, \mathcal{H}^{d}({\QZ}))$, $d=\dim X$. 
\end{abstract} 

\maketitle


\section{Introduction}\label{sec:intro}

Since the pioneering work of Artin and Mumford \cite{A-M}, 
unramified cohomology $H^{k}_{nr}(X, {\QZ})$ of a smooth complex projective variety $X$ with torsion coefficients 
has served as a powerful birational invariant for distinguishing rational varieties and 
nearly rational but irrational varieties. 
One of the key observations of Artin and Mumford was the identification of $H^2_{nr}(X, {\QZ})$, 
also known as the unramified Brauer group, with the torsion part of $H^3(X, {\Z})$ for unirational $X$. 
Unramified cohomology of the next degree $k=3$ and higher has been extensively studied 
beginning with Colliot-Th\'el\`ene and Ojanguren \cite{CT-O}. 
A breakthrough in degree $3$ was made by Colliot-Th\'el\`ene and Voisin \cite{CT-V}, 
which was made possible by the solution of the Bloch-Kato conjecture by Voevodsky and Rost \cite{Voe}, 
and where $H^{3}_{nr}(X, {\QZ})$ was identified with the defect $Z^4(X)$ 
of the integral Hodge conjecture for $H^4(X, {\Z})$. 
(Kahn \cite{Ka} noticed that the use of the Bloch-Kato conjecture can be reduced to that of the Merkurjev-Suslin theorem \cite{M-S}.) 
Later a description of the next group $H^{4}_{nr}(X, {\QZ})$ in terms of algebraic cycles 
was given in \cite{Vo}, \cite{Ma}. 

In degree $k\geq 5$, geometric interpretation of $H^{k}_{nr}(X, {\QZ})$ is mysterious. 
The purpose of this article is to look at a certain common structure in this series of work, 
which provides a geometric description of a relatively elementary piece of $H^{k}_{nr}(X, {\QZ})$ 
for general $k$. 
This would be a very first step for understanding the relationship between 
unramified cohomology and algebraic cycles. 

Recall (\cite{Gro}, \cite{B-O}) that the coniveau filtration of $H^{l}(X, {\Z})$ is defined by 
\begin{equation}\label{eqn: coniveau filtration}
N^{r}H^{l}(X, {\Z}) \: = \: 
\sum_{W\subset X} {\ker}( H^{l}(X, {\Z}) \to H^{l}(X\backslash W, {\Z}) ), 
\end{equation}
where $W$ ranges over all closed algebraic subsets of $X$ of codimension $\geq r$. 

Our result is stated as follows. 

\begin{theorem}\label{main1}
Let $X$ be a smooth complex projective variety. 
Let $k>1$. 
Assume that ${\rm CH}_{0}(X)$ is supported in dimension $<k$, 
e.g., $X$ rationally connected. 
Then $H^{k}_{nr}(X, {\QZ})$ has a filtration 
\begin{equation}\label{eqn: filtration}
0 \subset F_1 \subset F_2 \subset \cdots \subset F_{[k/2]} = H^{k}_{nr}(X, {\QZ}) 
\end{equation}
of length $[k/2]$, such that 
$F_1$ is isomorphic to the torsion part of the quotient 
$H^{k+1}(X, {\Z})/N^2H^{k+1}(X, {\Z})$. 
When $k$ is even, the next graded piece $F_2/F_1$ is isomorphic to 
a subquotient of the Griffiths group ${\rm Griff}^{k/2+1}(X)$. 
\end{theorem}

Although we have assumed that ${\rm CH}_{0}(X)$ is small, 
Theorem \ref{main1} is valid for general $X$ 
if we replace $H^{k}_{nr}(X, {\QZ})$ with the quotient of  
$H^{k}_{nr}(X, {\QZ})$ by $H^{k}_{nr}(X, {\Z})\otimes{\QZ}$; 
the group $H^{k}_{nr}(X, {\Z})$ vanishes when ${\rm CH}_0(X)$ is supported in dimension $<k$ (\cite{CT-V}). 
Since the study of $H^{k}_{nr}(X, {\QZ})$ has centered around varieties with small ${\rm CH}_{0}(X)$, 
we have chosen to state the above form for simplicity. 

When $k=2$, this is the result of Artin-Mumford \cite{A-M} where $N^2H^3(X, {\Z})=0$. 
When $k=3$, this is the result of Colliot-Th\'el\`ene and Voisin \cite{CT-V} 
where the torsion part of $H^4(X, {\Z})/N^2H^4(X, {\Z})$ coincides with the integral Hodge defect $Z^4(X)$. 
The torsion part of $H^{k+1}(X, {\Z})/N^{2}H^{k+1}(X, {\Z})$ is a natural generalization of 
the Artin-Mumford invariant and the Colliot-Th\'el\`ene--Voisin invariant. 
When $k=4$, an explicit description of $F_2/F_1$ as a subgroup of ${\rm Griff}^{3}(X)$ 
was given by Voisin \cite{Vo} assuming $F_1=0$; 
later the description of $F_1$ as above was given in \cite{Ma} Remark 4.2. 

The torsion part of $H^{k+1}(X, {\Z})/N^{2}H^{k+1}(X, {\Z})$ thus arises as 
a relatively elementary piece of $H^{k}_{nr}(X, {\QZ})$, and can be used as a first check for nontriviality.  
It vanishes when $X$ admits an integral cohomological decomposition of the diagonal 
in the sense of \cite{Vo2} (Proposition \ref{prop: diagonal decomposition}). 
The second part of Theorem \ref{main1} says that 
the next piece $F_{2}/F_{1}$ is controlled by ${\rm Griff}^{k/2+1}(X)$. 
Although this is not explicit, it says at least that 
triviality of ${\rm Griff}^{k/2+1}(X)$ implies that of $F_{2}/F_{1}$.  
When $k$ is odd, the role of the Chow group ${\rm CH}^{k/2+1}(X)$ will be partly 
replaced by the higher Chow group ${\rm CH}^{(k+3)/2}(X, 1)$: see \S \ref{ssec: odd degree}. 
 
The filtration $(F_{i})$ is introduced through 
the Bloch-Ogus spectral sequence with ${\Z}$-coefficients, 
and could be thought of as measuring a sort of geometric depth. 
In addition to the relationship with algebraic cycles, 
what makes the theory of unramified cohomology rich and useful is 
its description via Galois cohomology. 
When one constructs an element of $H^k_{nr}(X, {\QZ})$ by Galois cohomology, 
say in the Noether problem, 
it would be a natural question which level $F_i$ does that element belong to. 

As was first observed by Colliot-Th\'el\`ene and Voisin \cite{CT-V}, 
the $\mathcal{H}$-cohomology 
$H^{d-k}(X, \mathcal{H}^{d}({\QZ}))$ where $d=\dim X$
is also a birational invariant and is somewhat analogous to $H^{k}_{nr}(X, {\QZ})$. 
We also prove the following analogue of Theorem \ref{main1}. 

\begin{theorem}\label{main2}
Let $X$ be a smooth complex projective variety of dimension $d$. 
Let $k>1$. 
Assume that ${\rm CH}_0(X)$ is supported in dimension $<k-1$. 
Then $H^{d-k}(X, {\HH}^{d}({\QZ}))$ has a filtration of length $[k/2]$ 
\begin{equation*}\label{eqn: filtration homology intro}
0 \subset F_1 \subset F_2 \subset \cdots \subset F_{[k/2]} = H^{d-k}(X, {\HH}^{d}({\QZ})), 
\end{equation*}
such that $F_1$ is isomorphic to the finite abelian group 
\begin{equation}\label{eqn: 1st filter homology}
H^{2d-k+1}(X, {\Z})/N^{d-k+2}H^{2d-k+1}(X, {\Z}). 
\end{equation}
When $k$ is even, the next graded piece $F_2/F_1$ is isomorphic to 
a subquotient of the Griffiths group ${\rm Griff}_{k/2-1}(X)$. 
\end{theorem}

%
When $k=3$, Theorem \ref{main2} is the result of \cite{CT-V} 
where the group \eqref{eqn: 1st filter homology} coincides with the defect $Z_{2}(X)$ 
of the integral Hodge conjecture for $H^{2d-2}(X, {\Z})$. 
When $k=4$, Theorem \ref{main2} is the result of \cite{Ma} where 
$F_{2}/F_{1}$ coincides with the whole ${\rm Griff}_{1}(X)$ (when ${\rm CH}_0(X)$ is small). 
The group \eqref{eqn: 1st filter homology} is a generalization of $Z_{2}(X)$, 
and is also an analogue of the torsion part of $H^{k+1}(X, {\Z})/N^{2}H^{k+1}(X, {\Z})$. 
We notice that the filtration of $H^k_{nr}(X, {\QZ})$ and that of $H^{d-k}(X, {\HH}^{d}({\QZ}))$ 
have the same length. 

Theorem \ref{main2} holds for general $X$ without assumption on ${\rm CH}_{0}(X)$ 
if we replace $H^{d-k}(X, {\HH}^{d}({\QZ}))$ by the quotient 
\begin{equation*}
H^{d-k}(X, {\HH}^{d}({\QZ})) / (H^{d-k}(X, {\HH}^{d}({\Z}))\otimes {\QZ}), 
\end{equation*} 
and replace the group \eqref{eqn: 1st filter homology} (no longer finite) 
by its torsion part.

\begin{remark}
After the preprint version of this paper had appeared, 
independent work of Schreieder (\cite{Sc1}, \cite{Sc2}) appeared 
where filtrations closely related to \eqref{eqn: filtration} are introduced without using the Bloch-Kato conjecture 
and the relation between unramified cohomology and algebraic cycles is further investigated. 
\end{remark}


\textbf{Notation.} 
If $A$ is an abelian group and $n$ is a natural number, 
we write ${}_{n}A$ for the $n$-torsion part of $A$, 
namely $\{ x\in A \: | \: nx=0 \}$, 
and also write $A/n=A/nA$. 
We denote by ${}_{tor}A$ the torsion part of $A$, 
namely the union of all ${}_{n}A$, $n>0$.


\section{Bloch-Ogus theory}\label{sec:Bloch-Ogus}

In this section we recall the Bloch-Ogus theory following \cite{B-O} and \cite{CT-V}. 
Let $X$ be a smooth complex projective variety. 
Let $A$ be an abelian group (typically ${\Z}$ or ${\Z}/n$ or ${\QZ}$). 
We write ${\HH}^{k}(A)$ for the Zariski sheaf on $X$ associated to the presheaf 
$U\mapsto H^{k}(U, A)$. 
This is the $k$-th higher direct image of the constant sheaf $A$ 
by the natural map $X_{cl}\to X_{Zar}$ from classical topology to Zariski topology. 
The unramified cohomology $H^{k}_{nr}(X, A)$ of degree $k$ with coefficients in $A$ 
is defined as the Zariski cohomology 
\begin{equation*}
H^{k}_{nr}(X, A) = H^{0}(X, {\HH}^{k}(A)). 
\end{equation*}

Bloch-Ogus \cite{B-O} computed the $E_{2}$-page of the coniveau spectral sequence (\cite{Gro}) 
and obtained a first-quadrant spectral sequence 
\begin{equation*}
E_{2}^{p,q}=H^{p}(X, {\HH}^{q}(A)) \; \Longrightarrow \; H^{p+q}(X, A). 
\end{equation*}
This coincides with the Leray spectral sequence for $X_{cl}\to X_{Zar}$, 
and converges to the coniveau filtration of $H^{p+q}(X, A)$ as defined in \eqref{eqn: coniveau filtration}. 
It was proved in \cite{B-O} that 
\begin{itemize}
\item $H^{p}(X, {\HH}^{q}(A))=0$ when $p>q$, and  
\item when $A={\Z}$ and $p=q$, 
$E_{2}^{p,p}=H^{p}(X, {\HH}^{p}({\Z}))$ is isomorphic to the group $CH^{p}(X)/\!\sim_{alg}$ 
of codimension $p$ cycles modulo algebraic equivalence. 
The edge map 
$E_{2}^{p,p}\to H^{2p}(X, {\Z})$ coincides with the cycle map. 
\end{itemize}

When $A={\Z}$, the sheaf ${\HH}^{q}({\Z})$ is torsion-free (\cite{CT-V}, \cite{BV}), 
as a consequence of the Bloch-Kato conjecture proved by Voevodsky and Rost \cite{Voe}. 
This implies that (\cite{CT-V}), for each natural number $n$, 
we have a short exact sequence of sheaves 
\begin{equation*}
0 \to {\HH}^{q}({\Z}) \stackrel{n}{\to} {\HH}^{q}({\Z}) \to {\HH}^{q}({\Z}/n) \to 0, 
\end{equation*}
and thus a short exact sequence of groups 
\begin{equation}\label{eqn: UCT}
0 \to H^{p}(X, {\HH}^{q}({\Z}))/n \to 
H^{p}(X, {\HH}^{q}({\Z}/n)) \to {}_{n}H^{p+1}(X, {\HH}^{q}({\Z})) \to 0.  
\end{equation}
Colliot-Th\'el\`ene and Voisin \cite{CT-V} proved that 
when ${\rm CH}_{0}(X)$ is supported on an algebraic subset of $X$ of dimension $<k$, 
then $H^{k}_{nr}(X, {\Z})$ vanishes and 
$H^{d-k}(X, {\HH}^{d}({\Z}))$ is annihilated by some natural number 
where $d=\dim X$. 
In particular, $H^{d-k}(X, {\HH}^{d}({\Z}))\otimes {\QZ}$ also vanishes.


\section{The filtration}\label{sec:Hnr}

\subsection{Proof of Theorem \ref{main1}}\label{ssec: Proof 1.1}

In this subsection we prove Theorem \ref{main1}. 
Let $X$ be a smooth complex projective variety. 
Let $k>1$. 
By the exact sequence \eqref{eqn: UCT} with $(p, q)=(0, k)$, we have an isomorphism 
\begin{equation*}
H^{k}_{nr}(X, {\Zn})/(H^{k}_{nr}(X, {\Z})/n) \simeq 
{}_{n}H^{1}(X, {\HH}^{k}({\Z})).  
\end{equation*}
When ${\rm CH}_{0}(X)$ is supported in dimension $<k$, 
this reduces to an isomorphism 
$H^{k}_{nr}(X, {\Zn}) \simeq {}_{n}H^{1}(X, {\HH}^{k}({\Z}))$.   
We introduce a filtration on ${}_{n}H^{1}(X, {\HH}^{k}({\Z}))$ 
through the Bloch-Ogus spectral sequence with ${\Z}$-coefficients. 

In what follows, we write $(E_{r}^{p,q}, d_{r})$ for the $E_{r}$-page of 
the Bloch-Ogus spectral sequence with ${\Z}$-coefficients. 
In particular, $E_2^{1,k}=H^1(X, {\HH}^k({\Z}))$. 
The construction is based on the following elementary observation. 

\begin{lemma}\label{lem: calculation spectral sequence}
We have the inclusion 
\begin{equation}\label{eqn: filtration 1}
E^{1,k}_{\infty} = E^{1,k}_{[k/2]+1} \subset E^{1,k}_{[k/2]} 
\subset \cdots \subset E^{1,k}_{3} \subset E^{1,k}_{2}. 
\end{equation}
The differential $d_{r}$ embeds 
$E_{r}^{1,k}/E_{r+1}^{1,k}$ into $E_{r}^{r+1, k-r+1}$. 
When $r>(k+1)/4$, this gives a complex 
\begin{equation}\label{eqn: dr complex 1}
E_{r}^{1,k}/E_{r+1}^{1,k} \stackrel{d_{r}}{\hookrightarrow} E_{r}^{r+1, k-r+1} \twoheadrightarrow E_{\infty}^{r+1, k-r+1}. 
\end{equation}
\end{lemma}

\begin{proof}
Since $E_{r}^{p,q}=0$ in $p<0$, 
the differential $d_{r}$ that hits to $E_{r}^{1,k}$ is the zero map. 
(Note that $r\geq 2$.) 
This shows that  
\begin{equation*}
E_{r+1}^{1,k} = {\ker}(d_{r}: E_{r}^{1,k}\to E_{r}^{r+1,k-r+1}) 
\end{equation*}
is a subgroup of $E_{r}^{1,k}$, 
and $d_{r}$ induces an injective map 
$E_{r}^{1,k}/E_{r+1}^{1,k}\hookrightarrow E_{r}^{r+1, k-r+1}$. 
When $r>k/2$, 
we have $r+1>k-r+1$, 
and so $E_{r}^{r+1, k-r+1}=0$. 
Therefore 
$E_{r}^{1,k}=E_{r+1}^{1,k}=\cdots = E_{\infty}^{1,k}$. 

When $r>(k+1)/4$, we have $2r+1>k-2r+2$, 
so the next differential 
\begin{equation*}
d_{r} : E_{r}^{r+1,k-r+1}\to E_{r}^{2r+1,k-2r+2}=0 
\end{equation*}
is the zero map. 
This shows that 
$E_{r+1}^{r+1,k-r+1}=E_{r}^{r+1,k-r+1}/d_{r}(E_{r}^{1,k})$. 
Similarly, 
$E_{r+2}^{r+1,k-r+1}=E_{r+1}^{r+1,k-r+1}/d_{r+1}(E_{r+1}^{0,k+1})$ 
is a quotient of $E_{r+1}^{r+1,k-r+1}$ with 
$E_{r+2}^{r+1,k-r+1}=E_{\infty}^{r+1,k-r+1}$. 
This proves the last assertion. 
\end{proof}

\begin{remark}
By the proof, when $E_{2}^{0,k+1}=0$, 
e.g., when ${\rm CH}_{0}(X)$ is supported in dimension $<k+1$, 
the complex \eqref{eqn: dr complex 1} is exact, so we have an exact sequence 
\begin{equation*}
0 \to E_{r+1}^{1,k} \to E^{1,k}_{r} \to E_{r}^{r+1,k-r+1} \to E_{\infty}^{r+1,k-r+1} \to 0. 
\end{equation*} 
\end{remark}

Theorem \ref{main1} is deduced as follows. 

\begin{proof}[(Proof of Theorem \ref{main1})]
We take the $n$-torsion part of the filtration \eqref{eqn: filtration 1}. 
This gives the filtration 
\begin{equation}\label{eqn: filtration 2}
{}_{n}E^{1,k}_{\infty} = {}_{n}E^{1,k}_{[k/2]+1} \subset {}_{n}E^{1,k}_{[k/2]} 
\subset \cdots \subset {}_{n}E^{1,k}_{3} \subset {}_{n}E^{1,k}_{2} 
\end{equation}
on ${}_{n}E^{1,k}_{2}={}_{n}H^{1}(X, {\HH}^{k}({\Z}))$ with 
\begin{equation*}
{}_{n}E^{1,k}_{\infty} = 
{}_{n}(N^{1}H^{k+1}(X, {\Z})/N^{2}H^{k+1}(X, {\Z})) = 
{}_{n}(H^{k+1}(X, {\Z})/N^{2}H^{k+1}(X, {\Z})). 
\end{equation*}
Here the second equality holds because 
\begin{equation*}
H^{k+1}(X, {\Z})/N^{1}H^{k+1}(X, {\Z}) = E_{\infty}^{0,k+1} \subset H^{k+1}_{nr}(X, {\Z}) 
\end{equation*}
is torsion-free. 
The differential $d_{r}$ induces a map 
${}_{n}E_{r}^{1,k}/{}_{n}E_{r+1}^{1,k} \to {}_{n}E_{r}^{r+1, k-r+1}$. 
This is still injective because 
\begin{equation}\label{eqn: still injective}
{}_{n}E_{r+1}^{1,k} = {}_{n}E_{r}^{1,k}\cap E_{r+1}^{1,k} 
= {}_{n}E_{r}^{1,k} \cap {\rm Ker}(d_{r}). 
\end{equation}
Hence the complex \eqref{eqn: dr complex 1} gives rise to a complex 
\begin{equation*}\label{eqn: dr complex 2}
{}_{n}E_{r}^{1,k}/{}_{n}E_{r+1}^{1,k} \hookrightarrow 
{}_{n}E_{r}^{r+1, k-r+1} \to {}_{n}E_{\infty}^{r+1, k-r+1}. 
\end{equation*}

When $k=2p$ is even, we take $r=p$. 
Since $E_{p+1}^{1,2p}=E_{\infty}^{1,2p}$, 
this complex takes the form  
\begin{equation*}\label{eqn: dr complex 2}
{}_{n}E_{p}^{1,2p}/{}_{n}E_{\infty}^{1,2p} \hookrightarrow 
{}_{n}E_{p}^{p+1,p+1} \to {}_{n}N^{p+1}H^{2p+2}(X, {\Z}). 
\end{equation*}
Since the composition 
\begin{equation*}
CH^{p+1}(X)/\!\sim_{alg} \, = \, E_{2}^{p+1,p+1}\twoheadrightarrow E_{p}^{p+1,p+1} 
\twoheadrightarrow N^{p+1}H^{2p+2}(X, {\Z}) 
\end{equation*}
is the cycle map, we find that 
the embedding image of ${}_{n}E_{p}^{1,2p}/{}_{n}E_{\infty}^{1,2p}$ in 
${}_{n}E_{p}^{p+1,p+1}\subset E_{p}^{p+1,p+1}$ 
is contained in the image of 
${\rm Griff}^{p+1}(X)\subset E_{2}^{p+1,p+1}$ 
in $E_{p}^{p+1,p+1}$. 

Finally, we take the direct limit with respect to $n$. 
Then \eqref{eqn: filtration 2} gives the filtration 
\begin{equation*}\label{eqn: filtration 3}
{}_{tor}E^{1,k}_{\infty} = {}_{tor}E^{1,k}_{[k/2]+1} \subset {}_{tor}E^{1,k}_{[k/2]} 
\subset \cdots \subset {}_{tor}E^{1,k}_{3} \subset {}_{tor}E^{1,k}_{2} 
\end{equation*}
on 
\begin{equation*}
{}_{tor}E^{1,k}_{2} = {}_{tor}H^{1}(X, {\HH}^{k}({\Z})) = 
H^{k}_{nr}(X, {\QZ})/(H^{k}_{nr}(X, {\Z})\otimes {\QZ}) 
\end{equation*}
with 
\begin{equation*}
{}_{tor}E^{1,k}_{\infty} = {}_{tor}(H^{k+1}(X, {\Z})/N^{2}H^{k+1}(X, {\Z})). 
\end{equation*}
This is the filtration $(F_{i})$ in Theorem \ref{main1}. 
When $k=2p$ is even and $r=p$, 
the preceding paragraph shows that 
the injective image of 
${}_{tor}E^{1,k}_{p}/{}_{tor}E^{1,k}_{\infty}$ 
in $E^{p+1,p+1}_{p}$ is contained in the image of ${\rm Griff}^{p+1}(X)$, 
and hence is a subquotient of ${\rm Griff}^{p+1}(X)$.  
This proves Theorem \ref{main1}. 
\end{proof}

\begin{remark}
We also have a more direct filtration on $H^{k}_{nr}(X, {\Zn})$ 
through the Bloch-Ogus spectral sequence with ${\Zn}$-coefficients: 
\begin{equation*}
E^{0,k}_{\infty, {\Zn}} = E_{[(k+3)/2], {\Zn}}^{0,k} \subset \cdots \subset 
E_{3, {\Zn}}^{0,k} \subset E_{2, {\Zn}}^{0,k},  
\end{equation*}
where ${\Zn}$ in the subscript represents the coefficients. 
The first group $E^{0,k}_{\infty, {\Zn}}$ is 
$H^{k}(X, {\Zn})/N^{1}H^{k}(X, {\Zn})$, 
and when $k=2p-1$ is odd, the next graded piece is a subquotient 
of the kernel of 
$(CH^{p}(X)/\!\sim_{alg})/n \to H^{2p}(X, {\Zn})$. 

When $H^{k}_{nr}(X, {\Zn})\simeq {}_{n}H^{1}(X, {\HH}^{k}({\Z}))$, 
we have $E^{0,k}_{3,{\Zn}}\subset {}_{n}E_{3,{\Z}}^{1,k}$ 
because the second map in \eqref{eqn: UCT} 
anti-commutes with $d_{2}$ (see \cite{Ma} p.1510 -- p.1511).  
\end{remark}

\subsection{Higher Artin-Mumford invariant}

As explained in \S \ref{sec:intro}, 
the subgroup 
\begin{equation}\label{eqn: generalized AM inv}
{}_{tor}(H^{l}(X, {\Z})/N^{2}H^{l}(X, {\Z})) 
\end{equation}
of $H^{l-1}_{nr}(X, {\QZ})$ is a generalization of 
the Artin-Mumford invariant \cite{A-M} and the Colliot-Th\'el\`ene--Voisin invariant \cite{CT-V}. 
This can be nontrivial only when $l\leq \dim X+1$. 
An analogous invariant in the range $l>\dim X +1$ will be considered in \S \ref{sec:homology}. 

\begin{example}\label{ex: c.i.}
Let $X\subset {\proj}^N$ be a smooth complete intersection of dimension $d\geq 3$. 
The group \eqref{eqn: generalized AM inv} may be nontrivial only when 
(i) $l=d$ and (ii) $l=d+1$ with $d$ odd. 

Indeed, let $h=c_1(\mathcal{O}_{X}(1))$. 
By the Lefschetz hyperplane theorem,  
$H^{l}(X, {\Z})=0$ for odd $l<d$ and 
$H^{l}(X, {\Z})={\Z}h^{l/2}$ for even $l<d$. 
Thus $H^{l}(X, {\Z})/N^{2}H^{l}(X, {\Z})$ is trivial when $l<d$, $l\ne 2$,  
and has no torsion when $l=2$. 
When $l=d+1$, we have 
\begin{equation}\label{eqn: PD+Lef}
H^{d+1}(X, {\Z}) \simeq H_{d-1}(X, {\Z}) \simeq H_{d-1}({\proj}^N, {\Z}) 
\end{equation}
by Poincare duality and the Lefschetz hyperplane theorem for homology. 
Thus $H^{d+1}(X, {\Z})=0$ when $d$ is even. 
When $d$ is odd, this shows that 
$H^{d+1}(X, {\Z})$ is free of rank $1$ generated by ${\rm deg}(X)^{-1}h^{(d+1)/2}$. 
Since 
\begin{equation*}
{\Z}h^{(d+1)/2} \subset N^{(d+1)/2}H^{d+1}(X, {\Z}) \subset N^{2}H^{d+1}(X, {\Z}) \subset H^{d+1}(X, {\Z}), 
\end{equation*}
we see that 
$H^{d+1}(X, {\Z})/N^{2}H^{d+1}(X, {\Z})$ is cyclic of order dividing ${\rm deg}(X)$. 
It is the most transcendental part of the defect 
$H^{d+1}(X, {\Z})/N^{(d+1)/2}H^{d+1}(X, {\Z})$ of the integral Hodge conjecture for $H^{d+1}(X, {\Z})$. 
\end{example}

The invariant \eqref{eqn: generalized AM inv} is related to 
the cohomological integral decomposition of the diagonal (\cite{Vo2}) as follows. 
This is a generalization of \cite{Vo3} Proposition 4.9. 

\begin{proposition}\label{prop: diagonal decomposition}
If $X$ admits a cohomological integral decomposition of the diagonal, 
then $H^{l}(X, {\Z})/N^{2}H^{l}(X, {\Z})$ has no torsion. 
\end{proposition}

\begin{proof}
The assumption means that 
there exists a $d$-dimensional cycle $Z$ on $X\times D$ for some divisor $D\subset X$ such that  
\begin{equation*}\label{eqn: diagonal decomposition}
[\Delta_{X}] = [ \{ pt \} \times X ] + [Z] \qquad \textrm{in} \: \: H^{2d}(X\times X, {\Z})  
\end{equation*}
where $d=\dim X$ 
and $\Delta_{X} \subset X \times X$ is the diagonal. 
Let $j\colon \tilde{D}\to D\hookrightarrow X$ be a desingularization of $D$. 
As in the proof of \cite{Vo3} Proposition 4.9, 
we may assume that $Z$ lifts to a $d$-dimensional cycle $\tilde{Z}$ on $X \times \tilde{D}$ with 
\begin{equation}\label{eqn: diagonal decomposition 2}
[\Delta_{X}] = [ \{ pt \} \times X ] + ({\rm id}_{X}, j)_{\ast}[\tilde{Z}] \qquad \textrm{in} \: \: H^{2d}(X\times X, {\Z}).   
\end{equation}

Let $\alpha$ be an element of  $H^{l}(X, {\Z})$ which is torsion modulo $N^{2}H^{l}(X, {\Z})$. 
We let the correspondence \eqref{eqn: diagonal decomposition 2} act on $\alpha$. 
Then $[ \{ pt \} \times X ]$ annihilates $\alpha$. 
We shall show that $({\rm id}_{X}, j)_{\ast}[\tilde{Z}]$ sends 
$\alpha$ to an element of $N^2H^l(X, {\Z})$, 
which then implies $\alpha\in N^{2}H^{l}(X, {\Z})$. 

Let $\bar{\alpha}$ be the image of $\alpha$ in $H^{l}(X, {\Q})$. 
By the assumption we have $\bar{\alpha}\in N^2H^l(X, {\Q})$. 
Since rational coniveau filtration is functorial (\cite{A-K}), 
we see that 
\begin{equation*}
[\tilde{Z}]_{\ast}\bar{\alpha} \in N^1H^{l-2}(\tilde{D}, {\Q}) = N^1H^{l-2}(\tilde{D}, {\Z})\otimes {\Q}. 
\end{equation*}
This means that 
$[\tilde{Z}]_{\ast}\alpha \in H^{l-2}(\tilde{D}, {\Z})$ 
is torsion modulo $N^1H^{l-2}(\tilde{D}, {\Z})$. 
But since $H^{l-2}(\tilde{D}, {\Z})/N^1H^{l-2}(\tilde{D}, {\Z})$ is torsion-free, 
we find that 
$[\tilde{Z}]_{\ast}\alpha \in N^1H^{l-2}(\tilde{D}, {\Z})$. 
Therefore 
$j_{\ast}([\tilde{Z}]_{\ast}\alpha) \in N^2H^{l}(X, {\Z})$.  
\end{proof}

\subsection{Odd degree}\label{ssec: odd degree}

We give few remarks on the case when $k=2p-1$ is odd.  
In this case, the role of ${\rm Griff}^{p+1}(X)$ for controlling $F_{2}/F_{1}$ 
is replaced by the kernel of the edge map 
\begin{equation*}\label{eqn: edge map odd}
\alpha : H^{p}(X, {\HH}^{p+1}({\Z})) \to N^{p}H^{2p+1}(X, {\Z}) \subset H^{2p+1}(X, {\Z}). 
\end{equation*}
The group $H^{p}(X, {\HH}^{p+1}({\Z}))$ 
is related to the higher Chow group ${\rm CH}^{p+1}(X, 1)$ as follows. 

Recall (\cite{EV-MS}, \cite{Ke}) that we have a natural isomorphism 
\begin{equation*}
{\rm CH}^{p+1}(X, 1)\simeq H^{p}(X, \mathcal{K}^{M}_{p+1}), 
\end{equation*} 
where $\mathcal{K}^{M}_{p+1}$ is the Zariski sheaf associated to the Milnor $K$-theory. 
We also have homomorphisms of Zariski sheaves (see \cite{Es}, \cite{BV}) 
\begin{equation*}
\mathcal{K}_{p+1}^{M}\to {\HH}_{\mathcal{D}}^{p+1}(p+1) \to {\HH}^{p+1}({\Z}), 
\end{equation*}
where ${\HH}_{\mathcal{D}}^{q}(r)$ is the Deligne-Beilinson cohomology sheaf. 
This gives the maps 
\begin{equation*}
\xymatrix{
{\rm CH}^{p+1}(X, 1) \ar@{->>}[r] & H^{p}(X, {\HH}^{p+1}_{\mathcal{D}}(p+1)) \ar[r] \ar[d] & H^{p}(X, {\HH}^{p+1}({\Z})) \ar[d]^{\alpha} \\ 
           &    H^{2p+1}_{D}(X, {\Z}(p+1)) \ar[r]    &    H^{2p+1}(X, {\Z})
}
\end{equation*}
where the vertical maps are the edge maps in the respective spectral sequences. 
The composition 
${\rm CH}^{p+1}(X, 1) \to H^{2p+1}_{D}(X, {\Z}(p+1))$ 
is the so-called regulator map (see, e.g., \cite{MS} and the references therein), 
and the composition 
${\rm CH}^{p+1}(X, 1) \to H^{2p+1}(X, {\Z})$ 
gives the cycle map. 
The map ${\rm CH}^{p+1}(X, 1) \to H^{p}(X, {\HH}^{p+1}_{\mathcal{D}}(p+1))$ 
is surjective because 
the Gersten resolution of $\mathcal{K}_{p+1}^{M}$ and that of ${\HH}^{p+1}_{\mathcal{D}}(p+1)$ 
have the same last two terms. 

\begin{lemma}\label{lem: torsion cycle class}
The image of ${\rm CH}^{p+1}(X, 1) \to H^{2p+1}(X, {\Z})$ 
is a finite subgroup of $N^{p}H^{2p+1}(X, {\Z})$. 
\end{lemma}

\begin{proof}
The image of $H^{2p+1}_{D}(X, {\Z}(p+1)) \to H^{2p+1}(X, {\Z})$ is 
$F^{p+1}\cap H^{2p+1}(X, {\Z})$ where $(F^k)$ is the Hodge filtration of $H^{2p+1}(X, {\C})$. 
Since $F^{p+1}\cap \overline{F^{p+1}} = 0$, 
this is the torsion part of $H^{2p+1}(X, {\Z})$. 
\end{proof}

Let ${\rm CH}^{p+1}(X, 1)_{{\rm hom}}$ be the kernel of ${\rm CH}^{p+1}(X, 1) \to H^{2p+1}(X, {\Z})$. 
Then we obtain a map 
\begin{equation*}
\beta : {\rm CH}^{p+1}(X, 1)_{{\rm hom}} \to {\ker}(\alpha). 
\end{equation*} 
This is an analogue of 
${\rm CH}^{p+1}(X)_{{\rm hom}} \to {\rm Griff}^{p+1}(X)$, 
but may be no longer surjective in general as the following proposition shows.  
Let $A^{p+1}(X)\subset {\rm CH}^{p+1}(X)$ 
be the subgroup of cycles algebraically equivalent to zero. 
The Abel-Jacobi map on the torsion part of $A^{p+1}(X)$ takes the form 
\begin{equation*}
{}_{tor}AJ : {}_{tor}A^{p+1}(X) \to N^{p}H^{2p+1}(X, {\Z})\otimes {\QZ}. 
\end{equation*}

\begin{proposition}\label{prop: k odd higher chow}
We have a surjective map 
${\rm Coker}(\beta)\otimes {\QZ}\to {\ker}({}_{tor}AJ)$ 
with finite kernel. 
In particular, ${\rm Coker}(\beta)\otimes {\QZ} \ne 0$ if and only if ${\ker}({}_{tor}AJ)\ne 0$. 
\end{proposition}

\begin{proof}
We denote by $N$ the quotient of $N^pH^{2p+1}(X, {\Z})$ by its torsion part. 
Then $\alpha$ induces a surjective map 
$\bar{\alpha}\colon H^{p}(X, {\HH}^{p+1}({\Z})) \to N$, 
and the composition 
\begin{equation}\label{eqn: cycle map torsion}
{\rm CH}^{p+1}(X, 1) \to H^{p}(X, {\HH}^{p+1}({\Z})) \stackrel{\bar{\alpha}}{\to} N 
\end{equation}
is the zero map by Lemma \ref{lem: torsion cycle class}. 

We take the tensor product of \eqref{eqn: cycle map torsion} with ${\QZ}$. 
First note that 
$N\otimes {\QZ} = N^{p}H^{2p+1}(X, {\Z})\otimes {\QZ}$ 
and $\alpha\otimes {\QZ}=\bar{\alpha}\otimes {\QZ}$ 
because the torsion part is annihilated by $\otimes {\QZ}$. 
By \cite{Ma}, the first map of \eqref{eqn: cycle map torsion} induces a short exact sequence 
\begin{equation*}
0 \to {\rm CH}^{p+1}(X, 1)\otimes {\QZ} \to H^{p}(X, {\HH}^{p+1}({\Z})) \otimes {\QZ} \to {}_{tor}A^{p+1}(X) \to 0.   
\end{equation*}
Thus the fact that the composition \eqref{eqn: cycle map torsion} is zero implies that the map 
\begin{equation*}
\alpha \otimes {\QZ} : H^{p}(X, {\HH}^{p+1}({\Z})) \otimes {\QZ} \to N^{p}H^{2p+1}(X, {\Z})\otimes {\QZ} 
\end{equation*}
descends to a map from ${}_{tor}A^{p+1}(X)$, 
which coincides with the Abel-Jacobi map ${}_{tor}AJ$ (see \cite{CT} \S 4). 
Therefore we obtain the short exact sequence 
\begin{equation*}
0 \to {\rm CH}^{p+1}(X, 1)\otimes {\QZ} \to {\ker}(\bar{\alpha} \otimes {\QZ}) \to {\ker}({}_{tor}AJ) \to 0. 
\end{equation*}
We note that 
${\ker}(\bar{\alpha} \otimes {\QZ}) = {\ker}(\bar{\alpha}) \otimes {\QZ}$ 
because $N$ is free so that  
${\rm Tor}_1(N, {\QZ})=0$, and hence  
\begin{equation*}
0 \to {\ker}(\bar{\alpha}) \otimes {\QZ} \to H^{p}(X, {\HH}^{p+1}({\Z})) \otimes {\QZ} \stackrel{\bar{\alpha}\otimes{\QZ}}{\to} N\otimes {\QZ} \to 0 
\end{equation*}
remains exact. 

The natural maps 
${\rm Ker}(\alpha)\otimes {\QZ} \to {\rm Ker}(\bar{\alpha})\otimes {\QZ}$ 
and 
${\rm CH}^{p+1}(X, 1)_{{\rm hom}}\otimes {\QZ} \to {\rm CH}^{p+1}(X, 1)\otimes {\QZ}$ 
are surjective with finite kernels, 
because ${\rm Ker}(\alpha)\subset {\rm Ker}(\bar{\alpha})$ and 
${\rm CH}^{p+1}(X, 1)_{{\rm hom}} \subset {\rm CH}^{p+1}(X, 1)$ 
are subgroups with finite quotients.  
These two maps induce the following commutative diagram with exact rows: 
\begin{equation*}
\xymatrix{
0 \ar[r] & {\rm CH}^{p+1}(X, 1)\otimes {\QZ} \ar[r] & {\ker}(\bar{\alpha}) \otimes {\QZ} \ar[r] & {\ker}({}_{tor}AJ) \ar[r] & 0 \\ 
& {\rm CH}^{p+1}(X, 1)_{{\rm hom}} \otimes {\QZ} \ar[r]_{\beta\otimes{\QZ}}  \ar@{->>}[u] &  {\ker}(\alpha) \otimes {\QZ}  \ar@{->>}[u] \ar[r] & 
{\rm Coker}(\beta \otimes {\QZ})  \ar@{->>}[u] \ar[r] & 0 
}
\end{equation*}
Since ${\rm Coker}(\beta)\otimes {\QZ} = {\rm Coker}(\beta\otimes {\QZ})$, 
we obtain the surjective map 
${\rm Coker}(\beta)\otimes {\QZ} \to {\ker}({}_{tor}AJ)$ 
as desired. 
By the snake lemma for the above diagram, 
we see that the kernel of this map is a quotient of the kernel of  
${\rm Ker}(\alpha)\otimes {\QZ} \to {\rm Ker}(\bar{\alpha})\otimes {\QZ}$, 
and hence finite. 
Finally, if ${\ker}({}_{tor}AJ)=0$, 
then ${\rm Coker}(\beta)\otimes {\QZ}$ is finite and divisible, and so is zero. 
\end{proof}

The torsion Abel-Jacobi map ${}_{tor}AJ$ on ${}_{tor}A^{p+1}(X)$ 
is known to be injective when $p+1=2$ (\cite{M-S}, \cite{Mu}) and $p+1=\dim X$ (\cite{Ro}). 
When $p+1=\dim X -1$, the kernel ${\ker}({}_{tor}AJ)$ itself is a birational invariant (\cite{Vo4} \S 2.3.3).


\section{Proof of Theorem \ref{main2}}\label{sec:homology}

In this section we prove Theorem \ref{main2}. 
Since the argument is similar to the proof of Theorem \ref{main1}, we will be brief. 

Let $X$ be a smooth complex projective variety of dimension $d$. 
Let $k>1$. 
By the exact sequence \eqref{eqn: UCT} with $(p, q)=(d-k, d)$, we have 
\begin{equation*}
H^{d-k}(X, {\HH}^{d}({\Z}/n)) / (H^{d-k}(X, {\HH}^{d}({\Z}))/n) 
\simeq {}_{n}H^{d-k+1}(X, {\HH}^{d}({\Z})).   
\end{equation*}
We introduce a filtration on 
${}_{n}H^{d-k+1}(X, {\HH}^{d}({\Z}))$ 
through the Bloch-Ogus spectral sequence with ${\Z}$-coefficients, 
which we again write $(E_{r}^{p,q}, d_{r})$ as in \S \ref{ssec: Proof 1.1}. 
Lemma \ref{lem: calculation spectral sequence} is replaced by the following, 
which slightly differs from Lemma \ref{lem: calculation spectral sequence} 
in the exactness of \eqref{eqn: dr complex homology}. 

\begin{lemma}
We have the filtration of length $[k/2]$ 
\begin{equation}\label{eqn: filtration homology}
E^{d-k+1,d}_{\infty} = E^{d-k+1,d}_{[k/2]+1} \subset \cdots \subset E^{d-k+1,d}_{3} \subset E^{d-k+1,d}_{2} 
\end{equation}
on $E^{d-k+1,d}_{2}$. 
The differential $d_{r}$ embeds 
$E^{d-k+1,d}_{r}/E^{d-k+1,d}_{r+1}$ into $E^{d-k+r+1,d-r+1}_{r}$, 
which when $r>(k+1)/4$ is extended to a short exact sequence 
\begin{equation}\label{eqn: dr complex homology}
0 \to E^{d-k+1,d}_{r}/E^{d-k+1,d}_{r+1} \to 
E^{d-k+r+1,d-r+1}_{r} \to E^{d-k+r+1,d-r+1}_{\infty} \to 0. 
\end{equation}
\end{lemma}

\begin{proof}
The proof is similar to that of Lemma \ref{lem: calculation spectral sequence}. 
The point is that 
${\HH}^{q}({\Z})=0$ if $q>d$, 
because smooth affine varieties of dimension $d$ have 
homotopy type of CW complex of real dimension $\leq d$. 
Therefore $E_{2}^{p,q}=0$ if $q>d$. 
In particular, no nonzero $d_{r}$ hits to $E_{r}^{d-k+1,d}$. 
Hence 
\begin{equation*}
E^{d-k+1,d}_{r+1} = {\ker} ( d_{r} : E^{d-k+1,d}_{r} \to E^{d-k+r+1,d-r+1}_{r}). 
\end{equation*}
The sequence \eqref{eqn: dr complex homology} is exact because 
\begin{equation*}
E^{d-k+r+1,d-r+1}_{\infty} = E^{d-k+r+1,d-r+1}_{r+1} = E^{d-k+r+1,d-r+1}_{r}/d_{r}(E^{d-k+1,d}_{r})
\end{equation*}
when $r>(k+1)/4$. 
\end{proof}

\begin{proof}[(Proof of Theorem \ref{main2})]
Taking the $n$-torsion part of \eqref{eqn: filtration homology} gives the filtration 
\begin{equation*}\label{eqn: filtration homology 2}
{}_{n}E^{d-k+1,d}_{\infty} = {}_{n}E^{d-k+1,d}_{[k/2]+1} 
\subset \cdots \subset {}_{n}E^{d-k+1,d}_{3} \subset {}_{n}E^{d-k+1,d}_{2} 
\end{equation*}
on ${}_{n}E^{d-k+1,d}_{2}={}_{n}H^{d-k+1}(X, {\HH}^{d}({\Z}))$ 
with 
\begin{equation*}
{}_{n}E^{d-k+1,d}_{\infty} = {}_{n}(H^{2d-k+1}(X, {\Z}) / N^{d-k+2}H^{2d-k+1}(X, {\Z})). 
\end{equation*}
Here note that 
\begin{equation*}
N^{d-k+1}H^{2d-k+1}(X, {\Z})=H^{2d-k+1}(X, {\Z}) 
\end{equation*}
because the Bloch-Ogus spectral sequence is supported in the range $q\leq d$. 
The exact sequence \eqref{eqn: dr complex homology} gives a complex 
\begin{equation*}
{}_{n}E^{d-k+1,d}_{r} / {}_{n}E^{d-k+1,d}_{r+1} \hookrightarrow  
{}_{n}E^{d-k+r+1,d-r+1}_{r} \to {}_{n}E^{d-k+r+1,d-r+1}_{\infty}. 
\end{equation*}
Here the first map is injective for the same reason as \eqref{eqn: still injective}. 
When $k=2p$ and $r = p$, 
since $E^{d-2p+1,d}_{p+1}=E^{d-2p+1,d}_{\infty}$, 
this embeds 
${}_{n}E^{d-2p+1,d}_{p} / {}_{n}E^{d-2p+1,d}_{\infty}$  
into the kernel of 
$E^{d-p+1,d-p+1}_{p} \twoheadrightarrow E^{d-p+1,d-p+1}_{\infty}$. 
This kernel is a quotient of the kernel of 
$E^{d-p+1,d-p+1}_{2} \twoheadrightarrow E^{d-p+1,d-p+1}_{\infty}$, 
namely the Griffiths group ${\rm Griff}^{d-p+1}(X)$. 

Taking the direct limit with respect to $n$, 
we obtain the filtration 
\begin{equation*}\label{eqn: filtration homology 2}
{}_{tor}E^{d-k+1,d}_{\infty} = {}_{tor}E^{d-k+1,d}_{[k/2]+1} 
\subset \cdots \subset {}_{tor}E^{d-k+1,d}_{3} \subset {}_{tor}E^{d-k+1,d}_{2} 
\end{equation*}
on 
\begin{eqnarray*} 
{}_{tor}E^{d-k+1,d}_{2} 
& = & 
{}_{tor}H^{d-k+1}(X, {\HH}^{d}({\Z})) \\ 
& \simeq & 
H^{d-k}(X, {\HH}^{d}({\QZ})) / (H^{d-k}(X, {\HH}^{d}({\Z}))\otimes {\QZ})  
\end{eqnarray*} 
with 
\begin{equation*}
{}_{tor}E^{d-k+1,d}_{\infty} = {}_{tor}(H^{2d-k+1}(X, {\Z}) / N^{d-k+2}H^{2d-k+1}(X, {\Z})). 
\end{equation*} 
When ${\rm CH}_{0}(X)$ is supported in dimension $<k$,  
the group $H^{d-k}(X, {\HH}^{d}({\Z}))$ is torsion (\cite{CT-V}) and hence 
annihilated by the tensor product with ${\QZ}$. 
Moreover, for the same reason, 
when ${\rm CH}_{0}(X)$ is supported in dimension $<k-1$, 
the group 
\begin{equation*}
H^{2d-k+1}(X, {\Z}) / N^{d-k+2}H^{2d-k+1}(X, {\Z})  \; \subset \; 
H^{d-k+1}(X, {\HH}^d({\Z})) 
\end{equation*}
is torsion and so is finite. 
\end{proof}

Let $1<k<d$. 
The subgroup 
\begin{equation}\label{eqn: generalized AM inv homology}
{}_{tor}(H^{2d-k+1}(X, {\Z}) / N^{d-k+2}H^{2d-k+1}(X, {\Z}))
\end{equation}
of $H^{d-k}(X, {\HH}^{d}({\QZ}))$ 
is an analogue of the higher Artin-Mumford invariant \eqref{eqn: generalized AM inv}. 

\begin{example}
Let $X\subset {\proj}^N$ be a smooth complete intersection of dimension $d\geq 3$. 
Then the group \eqref{eqn: generalized AM inv homology} may be nontrivial only when $k$ is odd, 
in which case it is cyclic of order dividing ${\rm deg}(X)$. 

Indeed, we have 
\begin{equation*}
H^{2d-k+1}(X, {\Z}) \simeq H_{k-1}(X, {\Z}) \simeq H_{k-1}({\proj}^N, {\Z}) 
\end{equation*}
like \eqref{eqn: PD+Lef}. 
This shows that $H^{2d-k+1}(X, {\Z})=0$ when $k$ is even; 
and when $k$ is odd, $H^{2d-k+1}(X, {\Z})$ is free of rank $1$ generated by 
${\rm deg}(X)^{-1}h^{d-(k-1)/2}$ where $h=c_{1}(\mathcal{O}_{X}(1))$. 
Like the case $l=d+1$ in Example \ref{ex: c.i.}, this implies that 
the group \eqref{eqn: generalized AM inv homology} is cyclic of order dividing ${\rm deg}(X)$. 
It is the most transcendental part of the defect of the integral Hodge conjecture for $H^{2d-k+1}(X, {\Z})$. 
\end{example}

\medskip

\noindent
\textbf{Acknowledgement.} 
I am grateful to the referees for valuable suggestions and comments 
which led us to improve this paper.


\end{document}